\documentclass[reqno,a4paper,12pt]{amsart}

\usepackage{xcolor}
\usepackage{amsmath,amsthm,amssymb,mathtools}
\usepackage{hyperref}

\usepackage[no-math]{fontspec}
\usepackage[frozencache=true,cachedir=minted-cache]{minted}

\newmintinline[lean]{lean4}{bgcolor=white}
\newminted[leancode]{lean4}{escapeinside=!!, 
                            breaklines, 
                            fontsize=\normalsize}
\usepackage[sort,nocompress]{cite}
\mathtoolsset{showonlyrefs}

\newtheorem{definition}{Definition}[section]
\newtheorem{theorem}[definition]{Theorem}
\newtheorem*{theorem*}{Theorem}
\newtheorem{lemma}[definition]{Lemma}

\newtheorem{proposition}[definition]{Proposition}

\def\N{{\mathbb N}}
\def\Z{{\mathbb Z}}
\def\R{{\mathbb R}}
\def\T{{\mathbb T}}
\def\C{{\mathbb C}}

\newcommand{\one}{\mathbf 1}

\newcommand{\Span}{\operatorname{span}}

\begin{document}

\title[Cantor measures with odd base do not admit Fourier frames]{Cantor measures with odd base do not admit Fourier frames}

\author[Jaume de Dios Pont]{Jaume de Dios Pont}
\address{Center for Data Science, New York University, New York, New York 10011, USA}
\email{jdedios@nyu.edu}

\author[Lukas Liehr]{Lukas Liehr}
\address{Department of Mathematics, Bar-Ilan University, Ramat-Gan 5290002, Israel}
\email{lukas.liehr@biu.ac.il}

\author[Mitchell A. Taylor]{Mitchell A. Taylor}
\address{ Department of Mathematics\\
ETH Z\"urich, R\"amistrasse 101, 8092 Z\"urich, Switzerland}
\email{mitchell.taylor@math.ethz.ch}

\date{\today}
\subjclass[2020]{28A80, 42B05, 46B15}
\keywords{Fourier frames, spectrality, Cantor measure}

\begin{abstract}
We prove that the Cantor measure with base $b$ does not admit a Fourier frame whenever $b > 1$ is an odd integer. In particular, this answers a question of Strichartz on the existence of a Fourier frame for the middle third Cantor measure. A formalization of our main result in Lean 4 is also provided.
\end{abstract}

\maketitle

\section{Introduction and result}

\subsection{}
A finite Borel measure $\mu$ on $\R$ is called \emph{spectral} if it admits an orthonormal Fourier basis, i.e., there exists a countable set $\Lambda \subset \R$ such that, after scaling, the system $E(\Lambda) = \{ e_\lambda : \lambda \in \Lambda \}$ with $e_\lambda(x)=e^{2\pi i \lambda x}$ is an orthonormal basis for $L^2(\mu)$. The system $E(\Lambda)$ is called a Fourier frame for
$L^2(\mu)$ if there exist constants $0<A\leq B<\infty$ such that
$$
    A\|f\|_{L^2(\mu)}^2
    \leq
    \sum_{\lambda\in\Lambda}
    \left|
        \int_\R f(x)e^{-2\pi i\lambda x}\,d\mu(x)
    \right|^2
    \leq
    B\|f\|_{L^2(\mu)}^2
$$
for every $f\in L^2(\mu)$. In this case, $\mu$ is said to be \emph{frame-spectral}.
Thus, spectral measures form a distinguished subclass of frame-spectral measures.
The basic example of a spectral measure is given by $\mathcal{L}|_{[0,1]}$, where $\mathcal{L}|_\Omega$ denotes the restriction of the Lebesgue measure to a measurable subset $\Omega \subset \R$.
The classification of spectrality of measures of the form $\mathcal{L}|_{\Omega}$ in terms of tiling properties of $\Omega$ is the well-known Fuglede problem in dimension one \cite{fuglede1974commuting}. On the other hand, the existence of frames for finite measures of the form $\mathcal{L}|_{\Omega}$ was settled by Nitzan, Olevskii and Ulanovskii \cite{nitzan2016exponential}, who proved that $\mathcal{L}|_{\Omega}$ is frame-spectral whenever $\Omega$ has finite measure.

\subsection{}
Considerable attention has been devoted to identifying measures that are spectral or frame-spectral beyond restrictions of the Lebesgue measure, particularly among self-similar measures such as Cantor measures. The study of spectrality for these measures was initiated by
Jorgensen and Pedersen \cite{JorgensenPedersen98}, who proved that Cantor measures $\mu_b$ with respect to an even base $b$ (see Section \ref{sec:2} for the precise definition) are spectral. On the other hand, if $b>1$ is odd, then every collection of pairwise orthogonal exponentials in $L^2(\mu_b)$ has cardinality at most two, therefore ruling out the existence of an exponential basis. Strichartz gave an alternative proof of completeness for the exponential bases constructed by Jorgensen and Pedersen \cite{Strichartz98}, and later extended the construction to broader classes of Cantor measures while developing mock Fourier series and transforms \cite{Strichartz00}. The possible spectra of the middle fourth Cantor measure and related measures were further analyzed through rooted-tree parametrizations in \cite{DutkayHanSun09,Dai16}.  Łaba and Wang developed a systematic theory of spectral Cantor measures based on finite Fourier matrices and compatible digit sets \cite{LabaWang02}, and the arithmetic classification of spectral Bernoulli convolutions was advanced by Hu and Lau \cite{HuLau08} and completed by Dai \cite{Dai12}; related classifications were obtained in \cite{DaiHeLai13,DaiHeLau14}. This line of research culminated in the theorem of Dutkay, Haussermann and Lai, which holds in arbitrary dimension: every Hadamard triple generates a spectral self-affine measure \cite{DutkayHaussermannLai19}. For a broader overview, we refer to the survey \cite{DutkayLaiWang17}.

\subsection{}
The result of Jorgensen and Pedersen on spectrality of Cantor measures leaves open the question of whether an odd-base Cantor measure can nevertheless admit a Fourier frame. In this context, a question posed by Strichartz \cite{Strichartz00} asks if the middle third Cantor measure is frame-spectral.
Despite various developments on the existence and nonexistence of Fourier frames for measures (see the discussion following Theorem \ref{thm:no-frame}), the odd-base Cantor measures remained outside both the known constructions and known obstruction theorems. The main result of the present paper settles this problem by showing that any Cantor measure with an odd base does not admit a Fourier frame.

\begin{theorem}\label{thm:no-frame}
    If $b > 1$ is an odd integer, then the Cantor measure $\mu_b$ does not admit a Fourier frame.
\end{theorem}

Combining Theorem~\ref{thm:no-frame} with the spectral construction of Jorgensen and Pedersen yields the following classification: $\mu_b$ admits a Fourier frame if and only if $b> 1$ is even. Consequently, the even/odd dichotomy for the measures $\mu_b$ persists not only at the level of orthonormal Fourier bases, but at the broader level of Fourier frames.

\subsection{}
In more general terms, the problem of Strichartz mentioned above is concerned with the question of whether a non-spectral fractal measure can nevertheless admit a Fourier frame.
Lai and Wang constructed singular fractal measures that have only finitely many mutually orthogonal exponentials, but nevertheless admit Fourier frames \cite{LaiWang17}. Picioroaga and Weber constructed exponential Parseval frames for the Cantor measure $\mu_4$ by means of representations of Cuntz algebras \cite{PicioroagaWeber17}.

Various general results have clarified the structure of measures admitting Fourier frames, such as necessary density conditions for frame spectra in terms of estimates on the  Beurling dimension  \cite{DutkayHanSunWeber11}. In particular, Dutkay, Emami and Lai proved that the middle third Cantor measure admits an exponential Riesz sequence of maximal possible Beurling dimension \cite{DutkayEmamiLai21}. 

Regarding the properties of the measure, Lai characterized compactly supported absolutely continuous measures admitting Fourier frames \cite{Lai11} and He, Lai, and Lau proved, among other results, the pure type property for measures that are frame-spectral \cite{HeLaiLau13}.

There are also nonexistence results for Fourier frames for other classes of singular measures: for sums of singular measures failing translational absolute continuity \cite{FuLai18}, for mixed measures containing a surface piece with a point of non-zero Gaussian curvature \cite{Lev18}, for surface measures on convex bodies with everywhere positive Gaussian curvature \cite{IosevichLaiLiuWyman22}, and for a generic class of Salem measures \cite{li2025fourier}.
For results on the more general notions of frame measures, almost-Parseval frame towers, and weighted Fourier frames, we refer to the works \cite{DutkayHanWeber14,DutkayLai14,LaiWang17,AnFuLai19,DutkayRanasinghe16}.

\subsection{Usage of Large Language Models}

Large language models played an important role in the development of this work. At an early stage, we considered the spaces $V_n$ of functions that are constant on level-$n$ Cantor sets and the orthogonal differences $W_n=V_n\ominus V_{n-1}$. The initial goal was to obtain a positive solution to the Strichartz problem by constructing finite sets $\Gamma_n\subset\R$ such that $E(\Gamma_n)$ would form a frame for $W_n$ with frame bounds independent of $n$. Motivated by Strichartz's compatible pair and tower formalism \cite{Strichartz00}, and by replacing the compatible even dilations in the one dimensional model by the odd dilation $3$, we were led to the translated ternary digit sets
$$
        \Xi_m=
        \left\{
        \frac34\sum_{r=0}^{m-1}\varepsilon_r3^r:
        \varepsilon_r\in\{0,1\}
        \right\},
        \qquad
        \Gamma_n(\tau_n)=\tau_n+\frac{3^n}{4}+\Xi_{n-1}.
$$
With the help of GPT-5.5, we analyzed why these candidates fail to provide a frame for $W_n$ with bounds independent of $n$. For arbitrarily large $n$, one can find $\lambda_*\in\Gamma_n$ and increasingly many $\lambda_j\in\Gamma_n$ such that $\lambda_*-\lambda_j=3^{k_j}$. Since, for $F(\xi)=|\widehat{\mu_3}(\xi)|^2$, self-similarity gives $F(3^k)=F(1)>0$, the necessary Bessel estimate
$
        \sup_{t\in\R}\sum_{\lambda\in\Lambda}F(t-\lambda)<\infty
$
fails for these constructions by taking $t=\lambda_*$.

This failed construction suggested the obstruction used in the final proof. Indeed, the upper frame inequality applied to the constant function gives the bound
$
        \sum_{\lambda\in\Lambda}F(\lambda)\le B,
$
while the lower frame inequalities on the spaces $W_n$ force a lower bound independent of the scale. More precisely, a suitable normalized polynomial $p_n$ whose original expression was subsequently simplified into a more concise equivalent form with the assistance of GPT-5.5, leads to
$$
        X_n(\lambda)
        =
        \frac1{b^2 2^{n-2}}
        F(\lambda/b^n)\sin^2(2\pi\lambda/b^n)
        \cos^4(\pi z_n(\lambda))
        |p_n(z_n(\lambda))|^2,
$$
and the frame inequalities imply the lower bound $\sum_{\lambda\in\Lambda}X_n(\lambda)\ge \frac{A^2}{2b^2B}$, from which the final contradiction is obtained.

In this work, GPT-5.5 in the ChatGPT web interface was used for mathematical exploration, while GPT-5.5 in Codex helped with the Lean formalization. The authors checked and rewrote every part of the paper influenced by LLM-generated material, as well as the statement of the Lean formalization and all definitions it depends on.

\section{Preliminaries}\label{sec:2}

\subsection{}
Throughout the article, we fix an odd integer $b > 1$. For $d\in\{0,1\}$, we define the affine contraction $T_d : \R \to \R$ by 
$$
T_d(x)=\frac{x+2d}{b}.
$$
We write $\Omega:=\{0,1\}^{\N}$, let $\nu$ be the fair Bernoulli measure on $\Omega$ and define $\pi : \Omega \to \R$ by
$$
\pi(\omega)=\sum_{j=1}^{\infty}\frac{2\omega_j}{b^j}.
$$
The image of $\pi$ is the Cantor set with base $b$, which we denote by $C_b$. It is well-known that $C_b \subset [0,2/(b-1)]$ is compact and $C_b = T_0(C_b) \cup T_1(C_b)$. The Cantor measure $\mu_b$ on $C_b$ is the push-forward measure of $\nu$ under the map $\pi$, i.e., $\mu_b = \pi_{\#}\nu$. Alternatively, $\mu_b$ is defined as the unique probability measure supported on $C_b$ such that $\mu_b$ satisfies the self-similarity condition
\begin{equation}\label{eq:self-similar}
        \mu_b=\frac12\,\mu_b\circ T_0^{-1}
             +\frac12\,\mu_b\circ T_1^{-1}.
\end{equation}
Equivalently, for every bounded Borel function $f:\R\to\C$, one has
\begin{equation}\label{eq:self-similar-integral}
        \int_{\R}f\,d\mu_b
        =
        \frac12\int_{\R}f(T_0x)\,d\mu_b(x)
        +
        \frac12\int_{\R}f(T_1x)\,d\mu_b(x).
\end{equation}
The existence and uniqueness of the invariant measure follows from the general theory of contractive iterated function systems \cite{hutchinson1981fractals}.

\subsection{}
If $\eta=(\eta_1,\ldots,\eta_m) \in \{0,1 \}^m$ we put
$$
        |\eta| :=m,\quad
        T_\eta := T_{\eta_1}\circ\cdots\circ T_{\eta_m},\quad
        C_\eta := T_\eta(C_b),
$$
with the convention that $T_\varnothing = \mathrm{Id}$ and $C_\varnothing=C_b$.
We define
$$
        a_\eta := \sum_{j=1}^{m}\frac{2\eta_j}{b^j}.
$$
Basic properties of the above measures and operators are recorded in the following lemma.
\begin{lemma}\label{prop:cantor-measure} The following statements hold.
\begin{enumerate}
    \item For every $m\in\N_0$, every $\eta\in\{0,1\}^m$ and every $x\in\R$, we have
\begin{equation}\label{eq:affine-cylinder}
        T_\eta(x)=a_\eta+b^{-|\eta|}x.
\end{equation}
\item For every $m \in \N_0 $, the sets $\{C_\eta:\eta\in\{0,1\}^m\}$ are pairwise disjoint. Moreover, for every $\eta\in\{0,1\}^m$ we have
\begin{equation}\label{eq:cylinder-mass-support}
        \mu_b(C_\eta)=2^{-|\eta|}.
\end{equation}
\item For every $m\in \N_0$, every $\eta\in\{0,1\}^m$ and every bounded
Borel function $f:\R\to\C$, the integral of $f$ over $C_\eta$ with respect to the Cantor measure $\mu_b$ is given by
\begin{equation}\label{eq:cylinder-integral}
        \int_{C_\eta}f\,d\mu_b
        =
        2^{-|\eta|}
        \int_{\R}f(T_\eta x)\,d\mu_b(x).
\end{equation}
\end{enumerate}
\end{lemma}

\begin{proof}
Property (1) is proved by induction on $m=|\eta|$. The case
$m=0$ is the identity map. If $\eta=(\eta_1,\ldots,\eta_m)$ with
$m\ge1$, then the induction hypothesis applied to
$\eta'=(\eta_2,\ldots,\eta_m)$ gives
$$
        T_{\eta'}(x)
        =
        \sum_{j=2}^{m}\frac{2\eta_j}{b^{j-1}}+b^{-(m-1)}x.
$$
Therefore
$$
        T_\eta(x)
        =
        T_{\eta_1}(T_{\eta'}x)
        =
        \sum_{j=1}^{m}\frac{2\eta_j}{b^j}+b^{-m}x
        =
        a_\eta+b^{-m}x.
$$

We next prove the pairwise disjointness. Let
$\eta=(\eta_1,\ldots,\eta_m)\in\{0,1\}^m$, $d\in\{0,1\}$ and write
$(\eta,d)=(\eta_1,\ldots,\eta_m,d)$. Since
$C_b\subset[0,2/(b-1)]$, property (1) gives
$$
        C_{(\eta,0)}
        \subset
        a_\eta+b^{-m}\left[0,\frac{2}{b(b-1)}\right],
$$
whereas
$$
        C_{(\eta,1)}
        \subset
        a_\eta+b^{-m}\left[\frac2b,\frac2b+\frac{2}{b(b-1)}\right].
$$
These two intervals are disjoint because $2/(b(b-1))<2/b$. Hence, at
each step, the two sets obtained by appending one more coordinate are
disjoint. Induction on the level gives pairwise disjointness of
$\{C_\eta:\eta\in\{0,1\}^m\}$.

For $\eta=(\eta_1,\ldots,\eta_m)$ define
$$
        [\eta]
        =
        \{\omega\in\Omega:\omega_j=\eta_j \text{ for }1\le j\le m\},
$$
so that $C_\eta=\pi([\eta])$. Since the level-$m$ sets $C_\eta$ are
pairwise disjoint, we have $\pi^{-1}(C_\eta)=[\eta]$. Thus
$$
        \mu_b(C_\eta)
        =
        \nu(\pi^{-1}(C_\eta))
        =
        \nu([\eta])
        =
        2^{-m}
        =
        2^{-|\eta|}.
$$

 It remains to prove the integral formula. The desired formula is the
integral form (see \cite[Lemma~2.5]{khalil2020singular}) of the identity
$$
        \mu_b|_{C_\eta}
        =
        2^{-|\eta|}(T_\eta)_\#\mu_b.
$$
 We give a direct proof of this property. To do so, we observe that the conditional distribution of
the tail coordinates after fixing the first $m$ coordinates equal to
$\eta_1,\ldots,\eta_m$ is again $\nu$. Also,
$$
        \pi(\eta_1,\ldots,\eta_m,\omega_1,\omega_2,\ldots)
        =
        T_\eta(\pi(\omega)).
$$
Hence, for every bounded Borel function $f$, we have
$$
        \int_{C_\eta}f\,d\mu_b
        =
        \int_{[\eta]}f(\pi(\omega))\,d\nu(\omega)
        =
        2^{-|\eta|}
        \int_{\Omega}f(T_\eta(\pi(\omega)))\,d\nu(\omega).
$$
Since $\mu_b=\pi_\#\nu$, the last integral equals
$\int_{\R}f(T_\eta x)\,d\mu_b(x)$.
\end{proof}

\section{Haar functions on Cantor sets}

We define the Fourier transform of the measure $\mu_b$ in the usual sense by
$$
\widehat{\mu_b}(\xi)=
        \int_{\R}e^{-2\pi i\xi x}\,d\mu_b(x).
$$
Moreover, we let $F(\xi)=|\widehat{\mu_b}(\xi)|^2$ be the modulus squared of the Fourier transform of $\mu_b$. The function $F$ satisfies the following property \cite[Section 7]{JorgensenPedersen98}.

\begin{lemma}\label{FS:Lemma}
    For every $x \in \R$ we have $F(bx)=\cos^2(2\pi x)F(x)$.
\end{lemma}

We also make use of the following elementary estimate whose proof we omit. 

\begin{lemma}\label{lemma:trig}
    For every $x \in \R$ and every odd integer $b > 1$ we have the estimate $$\cos^2 x\ge \frac1{b^2}\sin^2 x\,\cos^2(bx).$$
\end{lemma}

For $n\ge0$ define
$$
        V_n=\Span \left \{\one_{C_\eta}: \eta \in \{ 0,1 \}^n \right \},
$$
where $\one_A$ denotes the characteristic function of a measurable set $A$.
For $n\ge1$, set $W_n=V_n\ominus V_{n-1}$. 
If
$\eta=(\eta_1,\ldots,\eta_{n-1})\in\{0,1\}^{n-1}$ and
$d\in\{0,1\}$, let
$$
        (\eta,d)=(\eta_1,\ldots,\eta_{n-1},d)\in\{0,1\}^n.
$$
Thus, $C_{(\eta,0)}$ and $C_{(\eta,1)}$ are the two level-$n$ subsets
of $C_\eta$ obtained by fixing the next coordinate to be $0$ or $1$.
For $\eta\in\{0,1\}^{n-1}$ define
$$
        \psi_\eta
        =
        2^{(n-1)/2}\left(\one_{C_{(\eta,0)}}-\one_{C_{(\eta,1)}}\right).
$$
The functions $\psi_\eta$ are the normalized Haar functions on a Cantor set, see \cite[p.~502]{guo2014boundary}. For completeness, we record
the following elementary lemma.

\begin{lemma}
    For each $n\ge1$, the family
$
        \{\psi_\eta:\eta\in\{0,1\}^{n-1}\}
$
is an orthonormal basis of $W_n$ with respect to the inner product $\langle \cdot, \cdot \rangle_{L^2(\mu_b)}$.
\end{lemma}

\begin{proof}
There are $2^n$ pairwise disjoint sets $C_\eta$ with
$\eta\in\{0,1\}^n$, and each has positive $\mu_b$-measure. Hence, their
indicator functions are linearly independent and $\dim V_n=2^n$. For
$\eta\in\{0,1\}^{n-1}$, the set $C_\eta$ is the disjoint union of
$C_{(\eta,0)}$ and $C_{(\eta,1)}$, so $V_{n-1}\subset V_n$. Therefore,
$\dim W_n=2^n-2^{n-1}=2^{n-1}$.

For a fixed $\eta\in\{0,1\}^{n-1}$, we have
$$
        \|\psi_\eta\|_{L^2(\mu_b)}^2
        =
        2^{n-1}\bigl(\mu_b(C_{(\eta,0)})+\mu_b(C_{(\eta,1)})\bigr)
        =
        1.
$$
If $\eta\ne\eta'$, then the supports of $\psi_\eta$ and
$\psi_{\eta'}$ are disjoint, hence the functions are orthogonal. Moreover,
$\psi_\eta$ is orthogonal to $V_{n-1}$. Indeed, it is supported on
$C_\eta$ and
$$
        \int_{C_\eta}\psi_\eta\,d\mu_b
        =
        2^{(n-1)/2}
        \bigl(\mu_b(C_{(\eta,0)})-\mu_b(C_{(\eta,1)})\bigr)
        =
        0.
$$
Thus, the functions $\psi_\eta$ are orthonormal elements of $W_n$. Since
there are $2^{n-1}$ of them and $\dim W_n=2^{n-1}$, they form an
orthonormal basis of $W_n$.
\end{proof}

For each $m\ge0$ define
$$
        Q_m=
        \left\{
        \sum_{r=0}^{m-1}\varepsilon_r b^r:
        \varepsilon_0,\ldots,\varepsilon_{m-1}\in\{0,1\}
        \right\},
        \qquad Q_0=\{0\},
$$
so that $Q_m$ is the set of non-negative integers whose base-$b$ digits are
all either $0$ or $1$, with no non-zero digit in positions
$b^m,b^{m+1},\ldots$. Define $q_m:\{0,1\}^m\to Q_m$ by
$$
        q_m(\eta)
        =
        \sum_{j=1}^{m}\eta_j b^{m-j}
        , \quad \eta=(\eta_1,\ldots,\eta_m)\in\{0,1\}^m.
$$
Note that $q_m(\eta)$ is the integer whose base-$b$ digits are
$\eta_1,\ldots,\eta_m$.

\begin{lemma}\label{lma:q}
For every $n\in\N$, the map
\[
        q_{n-1} :\{0,1\}^{n-1}\to Q_{n-1}
\]
is a bijection. Moreover, for every $\eta\in\{0,1\}^{n-1}$, we have
\[
        b^n a_\eta=2b\,q_{n-1}(\eta).
\]
\end{lemma}

\begin{proof}
The first assertion is the uniqueness of the base-$b$ expansion with digits in
$\{0,1\}$. More explicitly, each $\eta\in\{0,1\}^{n-1}$ determines the
integer
$$
        q_{n-1}(\eta)=\sum_{j=1}^{n-1}\eta_jb^{n-1-j}\in Q_{n-1},
$$
and every element of $Q_{n-1}$ is obtained uniquely in this way.

To prove the claimed identity, we observe that if $\eta=(\eta_1,\ldots,\eta_{n-1})$, then
$$
        b^n a_\eta
        =
        \sum_{j=1}^{n-1}2\eta_j b^{n-j}
        =
        2b\sum_{j=1}^{n-1}\eta_j b^{n-1-j}
        =
        2b\,q_{n-1}(\eta).
$$
\end{proof}
For $n\ge1$ and $\lambda\in\R$, we define $w_n(\lambda)$ and $z_n(\lambda)$ via
$$
        w_n(\lambda)
        =
        F(\lambda/b^n)\sin^2(2\pi\lambda/b^n),
        \qquad
        z_n(\lambda)=\frac{2b\lambda}{b^n}\pmod1\in\T.
$$
We have the following proposition.
\begin{proposition}
\label{prop:haar-coefficient-identity}
Fix $n\ge1$. For each $q\in Q_{n-1}$, let $\eta(q)$ be the unique
element of $\{0,1\}^{n-1}$ satisfying $q_{n-1}(\eta(q))=q$ and define
$\psi_q=\psi_{\eta(q)}$. If
$$
        f=\sum_{q\in Q_{n-1}}c_q\psi_q\in W_n, \quad c_q \in \C,
$$
then, for every $\lambda\in\R$,
\begin{equation}\label{eq:coefficient-identity}
        |\langle f,e_\lambda\rangle|^2
        =
        \frac{w_n(\lambda)}{2^{n-1}}
        \left|
        \sum_{q\in Q_{n-1}}c_qe^{-2\pi iqz_n(\lambda)}
        \right|^2.
\end{equation}
\end{proposition}

\begin{proof}
Fix $\lambda\in\R$ and $\eta\in\{0,1\}^{n-1}$. The integral formula in Lemma~\ref{prop:cantor-measure}
over $C_{(\eta,d)}$, where $d\in\{0,1\}$, gives
$$
        \int_{C_{(\eta,d)}}e^{-2\pi i\lambda x}\,d\mu_b(x)
        =
        2^{-n}
        e^{-2\pi i\lambda(a_\eta+2d/b^n)}
        \widehat{\mu_b}(\lambda/b^n).
$$
Therefore
$$
        \langle\psi_\eta,e_\lambda\rangle
        =
        2^{(n-1)/2}2^{-n}
        e^{-2\pi i\lambda a_\eta}
        \widehat{\mu_b}(\lambda/b^n)
        \left(1-e^{-2\pi i\,2\lambda/b^n}\right).
$$
Taking absolute values gives
$$
        |\langle\psi_\eta,e_\lambda\rangle|^2
        =
        2^{-(n-1)}
        F(\lambda/b^n)\sin^2(2\pi\lambda/b^n)
        =
        2^{-(n-1)}w_n(\lambda).
$$
Furthermore, by Lemma \ref{lma:q}, we have
$$
        e^{-2\pi i\lambda a_\eta}
        =
        e^{-2\pi i q_{n-1}(\eta)z_n(\lambda)}.
$$
Now, we sum the preceding formula with weights $c_q$,
using $\eta=\eta(q)$ for each $q\in Q_{n-1}$. The common factor has
squared modulus $2^{-(n-1)}w_n(\lambda)$ and the remaining phase sum is
exactly
$$
        \sum_{q\in Q_{n-1}}c_qe^{-2\pi iqz_n(\lambda)}.
$$
Squaring and taking absolute value proves the desired identity.
\end{proof}

We write $\T=\R/\Z$ and denote by $m_\T$ the normalized Haar measure on $\T$.

\begin{proposition}\label{prop:sampling}
Assume that $E(\Lambda)$ is a Fourier frame for
$L^2(\mu_b)$ with frame bounds $A,B$. For $n\ge1$ define the positive measure
$$
        \sigma_n
        =
        \frac1{2^{n-1}}
        \sum_{\lambda\in\Lambda}
        w_n(\lambda)\delta_{z_n(\lambda)}
$$
on $\T$. Then $\sigma_n(\T)\le B$. Moreover, for every trigonometric
polynomial
$$
        P(z)=\sum_{q\in Q_{n-1}}c_qe^{-2\pi iqz}, \quad c_q \in \C,
$$
we have
\begin{equation}\label{eq:sampling-frame}
        A\|P\|_{L^2(\T,m_\T)}^2
        \le
        \int_{\T}|P(z)|^2\,d\sigma_n(z)
        \le
       B\|P\|_{L^2(\T,m_\T)}^2.
\end{equation}
\end{proposition}

\begin{proof}
Fix any $\eta\in\{0,1\}^{n-1}$. By the proof of the preceding
proposition, we have
$$
        |\langle\psi_\eta,e_\lambda\rangle|^2
        =
        2^{-(n-1)}w_n(\lambda).
$$
Hence,
$$
        \sigma_n(\T)
        =
        \frac1{2^{n-1}}\sum_{\lambda\in\Lambda}w_n(\lambda)
        =
        \sum_{\lambda\in\Lambda}|\langle\psi_\eta,e_\lambda\rangle|^2
        \le B,
$$
where the last inequality is the upper frame bound applied to the unit
vector $\psi_\eta$.
Let
$$
        f=\sum_{q\in Q_{n-1}}c_q\psi_q\in W_n.
$$
Since the functions $\psi_q$ form an orthonormal basis of $W_n$, we have
$$
        \|f\|_{L^2(\mu_b)}^2=\sum_{q\in Q_{n-1}}|c_q|^2.
$$
The frequencies in $Q_{n-1}$ are distinct integers, so orthogonality in
$L^2(\T,m_\T)$ gives
$$
        \|P\|_{L^2(\T,m_\T)}^2=\sum_{q\in Q_{n-1}}|c_q|^2.
$$
Finally, the coefficient identity in Proposition~\ref{prop:haar-coefficient-identity} allows us to write
$$
        \sum_{\lambda\in\Lambda}|\langle f,e_\lambda\rangle|^2
        =
        \frac1{2^{n-1}}\sum_{\lambda\in\Lambda}
        w_n(\lambda)|P(z_n(\lambda))|^2
        =
        \int_{\T}|P(z)|^2\,d\sigma_n(z).
$$
Applying the frame inequalities to $f$ proves the desired two-sided
estimate.
\end{proof}

\section{The contradiction argument}
For $m\ge0$ define
$$
        D_m(z)=\sum_{q\in Q_m}e^{-2\pi iqz},
$$
and, for $n\ge2$, define
$$
        p_n(z)=2^{-(n-2)/2}D_{n-2}(bz).
$$
Basic properties of these functions are recorded in the following proposition.
\begin{proposition}\label{prop:test-polynomials}
For every $n\ge2$, we have
\begin{equation*}\label{eq:test-norms}
        \|p_n\|_{L^2(\T,m_\T)}=1,
        \quad
        \|(1+e^{-2\pi iz})p_n\|_{L^2(\T,m_\T)}^2=2
\end{equation*}
and
\begin{equation}\label{eq:pn-product}
        |p_n(z)|^2
        =
        2^{n-2}\prod_{r=1}^{n-2}\cos^2(\pi b^rz),
\end{equation}
where the empty product is interpreted as $1$. Moreover, for every $\lambda \in \R$, if we define the quantities $t,z$ via $t=\lambda/b^n$ and $z=z_n(\lambda)$, then
\begin{equation}\label{eq:F-product}
        F(\lambda)
        =
        F(t)\cos^2(2\pi t)\cos^2(\pi z)
        \prod_{r=1}^{n-2}\cos^2(\pi b^rz).
\end{equation}
\end{proposition}

\begin{proof}
The frequencies of $p_n$ are the elements of $bQ_{n-2}$, and the
frequencies of $e^{-2\pi iz}p_n$ are the elements of $1+bQ_{n-2}$.
These two sets are disjoint subsets of $Q_{n-1}$. Since distinct integer
frequencies are orthogonal in $L^2(\T,m_\T)$, we get
$$
        \|p_n\|_{L^2(\T,m_\T)}^2
        =
        2^{-(n-2)}|Q_{n-2}|
        =
        1
$$
and
$$
        \|(1+e^{-2\pi iz})p_n\|_{L^2(\T,m_\T)}^2
        =
        2.
$$
Note that we may factor
$$
        D_m(z)=\prod_{r=0}^{m-1}\left(1+e^{-2\pi ib^rz}\right).
$$
Taking $m=n-2$, replacing $z$ by $bz$ and using
$
        |1+e^{-2\pi iu}|^2=4\cos^2(\pi u)
$
gives
$$
        |p_n(z)|^2
        =
        2^{-(n-2)}
        \prod_{r=0}^{n-3}4\cos^2(\pi b^{r+1}z)
        =
        2^{n-2}\prod_{r=1}^{n-2}\cos^2(\pi b^rz).
$$

It remains to prove the product formula for $F$. Iterating Lemma~\ref{FS:Lemma} gives
$$
        F(\lambda)
        =
        F(t)\prod_{r=0}^{n-1}\cos^2(2\pi b^rt).
$$
Since $z=2bt\pmod1$, the factor with $r=1$ is
$\cos^2(\pi z)$, while the factors with $r=2,\ldots,n-1$ are
$\cos^2(\pi b^{r-1}z)$. This is the displayed formula.
\end{proof}

We are now ready to prove the main result of this paper.

\begin{proof}[Proof of Theorem \ref{thm:no-frame}]
Suppose, to the contrary, that $E(\Lambda)=\{e_\lambda:\lambda\in\Lambda\}$
is a Fourier frame for $L^2(\mu_b)$ with bounds $0<A\le B<\infty$.

Fix $n\ge2$. The frequencies of $p_n$ lie in $bQ_{n-2}$, while the
frequencies of $e^{-2\pi iz}p_n$ lie in $1+bQ_{n-2}$. Both sets are
subsets of $Q_{n-1}$. Hence, both $p_n$ and
$(1+e^{-2\pi iz})p_n$ are of the form allowed by Proposition~\ref{prop:sampling}. Apply the estimate in this proposition to
$(1+e^{-2\pi iz})p_n$. By the first part of
Proposition \ref{prop:test-polynomials}, we have
$$
        2A
        \le
        \int_{\T}|1+e^{-2\pi iz}|^2|p_n(z)|^2\,d\sigma_n(z)
        =
        4\int_{\T}\cos^2(\pi z)|p_n(z)|^2\,d\sigma_n(z),
$$
and hence
\begin{equation}\label{eq:cos2-lower}
        \int_{\T}\cos^2(\pi z)|p_n(z)|^2\,d\sigma_n(z)\ge \frac A2.
\end{equation}
Applying the upper bound in Proposition~\ref{prop:sampling} to $p_n$ gives
\begin{equation}\label{eq:pn-upper}
        \int_{\T}|p_n(z)|^2\,d\sigma_n(z)\le B.
\end{equation}
Cauchy--Schwarz, applied with the finite positive measure $\tau_n$ defined by
$
        d\tau_n(z)=|p_n(z)|^2\,d\sigma_n(z),
$
then gives
$$
\left(
        \int_{\T}\cos^2(\pi z)\,d\tau_n(z)
\right)^2
\le
\tau_n(\T)
\left(
        \int_{\T}\cos^4(\pi z)\,d\tau_n(z)
\right).
$$
Together with the two preceding inequalities, this yields
\begin{equation}\label{eq:cos4-lower}
        \int_{\T}\cos^4(\pi z)|p_n(z)|^2\,d\sigma_n(z)
        \ge
        \frac{A^2}{4B}.
\end{equation}
For $n\ge2$ and $\lambda\in\Lambda$ define
$$
        X_n(\lambda)
        =
        \frac1{b^2 2^{n-2}}
        w_n(\lambda)
        \cos^4(\pi z_n(\lambda))
        |p_n(z_n(\lambda))|^2.
$$
Using the definition of $\sigma_n$ and \eqref{eq:cos4-lower}, we get
\begin{equation}\label{eq:X-lower}
\begin{aligned}
        \sum_{\lambda\in\Lambda}X_n(\lambda)
        &=
        \frac1{b^2 2^{n-2}}
        \sum_{\lambda\in\Lambda}
        w_n(\lambda)\cos^4(\pi z_n(\lambda))
        |p_n(z_n(\lambda))|^2       \\
        &=
        \frac2{b^2}
        \int_{\T}\cos^4(\pi z)|p_n(z)|^2\,d\sigma_n(z)
        \ge
        \frac{A^2}{2b^2B}.
\end{aligned}
\end{equation}

We next prove that the same sums tend to zero. Fix $n\ge2$,
$\lambda\in\Lambda$ and put $t=\lambda/b^n$ and $z=z_n(\lambda)$.
Lemma~\ref{lemma:trig}, applied with $x=2\pi t$, gives
$
        \cos^2(2\pi t)
        \ge
        \frac1{b^2}\sin^2(2\pi t)\cos^2(2\pi bt)
        =
        \frac1{b^2}\sin^2(2\pi t)\cos^2(\pi z).
$
Combining this inequality with \eqref{eq:pn-product} and
\eqref{eq:F-product}, we obtain the pointwise estimate
\begin{equation}\label{eq:X-domination}
        0\le X_n(\lambda)\le F(\lambda).
\end{equation}
The upper frame inequality applied to the constant function $1$ gives
\begin{equation}\label{eq:F-summable}
        \sum_{\lambda\in\Lambda}F(\lambda)
        =
        \sum_{\lambda\in\Lambda}|\langle 1,e_\lambda\rangle|^2
        \le B.
\end{equation}
For each fixed $\lambda$, the continuity of $\widehat{\mu_b}$ gives
$
        F(\lambda/b^n) \to F(0)=1,
$
while we also have
$
        \sin^2(2\pi\lambda/b^n)\to 0.
$
Hence, $w_n(\lambda)\to0$. Since $p_n$ is a normalized sum of
$2^{n-2}$ unimodular exponentials, we have
$
        |p_n(z)|^2\le 2^{n-2}
$
and therefore
$$
        0\le X_n(\lambda)\le \frac1{b^2}w_n(\lambda)\to 0,
        \quad n\to\infty,
$$
for every fixed $\lambda\in\Lambda$. By \eqref{eq:X-domination} and
\eqref{eq:F-summable}, dominated convergence on the countable set
$\Lambda$ yields
\begin{equation}\label{eq:X-vanishing}
        \sum_{\lambda\in\Lambda}X_n(\lambda)\to 0.
\end{equation}
This contradicts the uniform lower bound \eqref{eq:X-lower}. Therefore,
$\mu_b$ does not admit a Fourier frame.
\end{proof}

\section{Appendix: Lean formalization.}

This section discusses the Lean 4 \cite{lean4} formalization of Theorem \ref{thm:no-frame}. The source code of the formalization can be found in \url{https://github.com/jaumededios/Cantor_Measure_Frames}. The file \lean{Showcase.lean} therein provides a self-contained version of the main statement that was Lean verified. 

The statements in \lean{Showcase.lean} have been carefully curated and reviewed by the authors with the objective of making them understandable to a broad audience, but the rest of the Lean code (the files containing the Lean proof) have been generated by large language models. In this project, the goal of the Lean translation is simply to verify correctness -- the Lean proof is not meant to be digested by the reader.

This autoformalization relies heavily on the definitions of multiple mathematical objects in Mathlib \cite{mathlib}, which is the standard Lean mathematics library. In particular, the trust we placed in this formalization is possible only because of the enormous efforts by the Mathlib community.

The syntax of theorems in Lean takes the form
\begin{leancode}
theorem TheoremName -- Two dashes start a comment, (this is a comment)
    {variable1 : Type} (variable2 : Type) (variable3 : Type) : 
    -- The colon separates the hypotheses and the conclusion
    conclusion
\end{leancode}
Every variable in Lean must have a type. Types are a primitive notion, analogous to sets. In Lean, hypotheses are propositions and propositions are variables. For example, if a hypothesis is \emph{let $b>1$ be a natural number}, in Lean we would write \lean{{b : ℕ} (hb : 1 < b)}.

The main theorem of the paper (deliberately rewritten closer to the Lean analog) may be stated as follows.

\begin{theorem*}
    Let $b\in \mathbb N$ be an odd number larger than two. Let $\mu_b$ be the Cantor measure with base $b$.
    Let $F:= (f_1, \dots, f_n, \dots)$ be a family of complex exponentials (i.e.~$f_j(x) := e^{2\pi i k_j x}$ for some $k_j \in \R$). Then the family $F$ is not a frame. 
\end{theorem*}
In Lean, this theorem will translate to the following statement.
\begin{leancode}
theorem NoFourierFrameExists
    {b : ℕ} (hb : 1 < b) (hb_odd : Odd b)
    (F : ℕ → Lp ℂ 2 (cantor_μ b)) (hF : IsExpSystem F) :
    ¬ IsFrame F
\end{leancode}
The definition of \lean{IsFrame} (see Section \ref{sec:fourier_frames}) takes advantage of implicit variables. Indeed, since $F$ is defined as a family of functions in $L^2(\mu_b)$, one can infer the Hilbert space in question from the family $F$.

The above theorem contains three definitions that are not available in Mathlib and that had to be defined by us. 
The rest of this section explains how these definitions were built from the Mathlib primitives.

\begin{enumerate}
    \item \lean{cantor_μ b}, the Cantor measure $\mu_b$.
    \item \lean{IsExpSystem F}, which states that all elements of a family of functions in $L^2(\mu)$ with $\mu$ a measure on $\R$ are $\mu$-almost everywhere equal to a complex exponential function.
    \item \lean{IsFrame F}, which states that a family of functions $F$ is a frame. 
\end{enumerate}

\subsection{Definition of the Cantor measure \texorpdfstring{$\mu_b$}{mu b}}

We define the \emph{Cantor measure} with base $b>1$ as the pushforward of the uniform measure on  $\{0,1\}^{\mathbb N}$ by the map that sends  $\omega = (\omega_0, \dots , \omega_k, \dots) \in \{0,1\}^{\mathbb N_0}$ to the number 
$$
\sum_{n\in \mathbb N_0} \frac{2\cdot \omega_n}{b^{n+1}}
$$
Note the change in indexing from the rest of the paper, as in Lean the natural numbers start at zero.

To translate this to Lean, we start by defining an \emph{abbreviation}\footnote{There is a subtle difference between how Lean \emph{unfolds} abbreviations and definitions that makes using an abbreviation here more convenient. For the purposes of understanding the result, however, they are equivalent.}. In Lean, \lean{Fin 2} is the canonical type with two elements, which is canonically equivalent to  $\{0,1\}\subset \mathbb N$.
\begin{leancode}
/-- The set of functions from the Naturals to {0,1}. -/
abbrev Ω : Type := ℕ → Fin 2
\end{leancode}
We define the uniform measure over $\Omega$ by taking the infinite product of the uniform measures on $\{ 0,1 \}$. In order to define these uniform measures, Mathlib has a function \lean{PMF.uniformOfFintype}, which takes a finite type $\{ 0, \dots, n-1 \}$ and returns a \emph{probability mass function}. This \emph{probability mass function} must then be cast into a regular measure using \lean{.toMeasure}, the type that the infinite product of probability measures needs. 
\begin{leancode}
/-- The uniform product measure on Ω. -/
def uniformBoolSeq : Measure Ω :=
  Measure.infinitePi 
    (fun _ : ℕ ↦ (PMF.uniformOfFintype (Fin 2)).toMeasure)
\end{leancode}
To define the Cantor measure we must define the coding map $\sum_{n\in \mathbb N} \frac{2\cdot \omega_n}{b^{n+1}}$ as follows.
\begin{leancode}
/-- The coding map for the `{0, 2}` base-`b` Cantor measure. -/
def code (b : ℝ) (ω : Ω) : ℝ :=
  ∑' n, 2 * (ω n) / (b ^ (n + 1))
\end{leancode}
As a last step, we define the Cantor measure $\mu_b$ as the pushforward of the coding map of the uniform measure on Boolean sequences.
\begin{leancode}
/-- The base-`b` Cantor measure, using digits `{0, 2}`. -/
def cantor_μ (b : ℕ) : Measure ℝ :=
  Measure.map (code b) uniformBoolSeq
\end{leancode}

\subsection{Definition of Fourier frames}
\label{sec:fourier_frames}
The main result will split the definition of Fourier frames into an \lean{IsFrame} statement (stating the usual frame condition) and an \lean{IsExponential} statement (stating that each element of the frame is a complex exponential). First, we define a complex exponential with frequency $k \in \R$.

\begin{leancode}
/-- Definition of the complex exponential at frequency k. -/
def e (k : ℝ) : ℝ → ℂ := 
  fun x ↦ Complex.exp (2 * π * Complex.I * k * x)
\end{leancode}

Let $\mu$ be a measure on $\mathbb{R}$. A family of $L^p(\mu)$ functions $\{f_j\}_{j\in \mathbb{N}}$ is an exponential family if, for any $j$, there is some $k \in \R$ such that the function $f_j$ is $\mu$-almost everywhere equal to $e^{2\pi i k \cdot}$. In Lean, we may write this as follows.

\begin{leancode}
/-- A family of Lᵖ functions is an exponential system if they are
    μ-a.e. equal to exponential functions. -/
def IsExpSystem {μ : Measure ℝ} {p : ℝ≥0∞} (F : ℕ → Lp ℂ p μ) : Prop :=
   (∀ j : ℕ , ∃ k : ℝ, (F j) =ᵐ[μ] (e k) )
\end{leancode}

The next definition is that of a frame. Let $H$ be a Hilbert space. A family of elements $\{h_j\}_{j \in \iota} \subset H$ is a frame if there are constants $0<A,B<\infty$ such that for any $g\in H$ one has
$$
A \|g\|^2 \le \sum_{j \in \iota}|\langle g, h_j\rangle|^2 \le B \|g\|^2.
$$
In order to define this in Lean, one has to make some changes. Mathlib does not define Hilbert spaces directly, but instead specifies the properties that are needed. In this case, we define $H$ to be a normed additive group with an inner product structure\footnote{In order to define a Hilbert space, one would need to add the \lean{[CompleteSpace E]} condition, which is not needed to define a frame.}. Mathlib reserves $|\cdot|$ for the real absolute value, and uses $\|\cdot\|$ for the norm of a complex number. Moreover, Lean uses $∑'$ to denote sums over infinite sets. With these caveats in hand, the definition of a frame in Lean is as follows.
\begin{leancode}
def IsFrame {ι H : Type*} [NormedAddCommGroup H] [InnerProductSpace ℂ H]
    (atoms : ι → H) : Prop :=
  ∃ A > 0, ∃ B > 0,
    ∀ f : H,
      A * ‖f‖ ^ 2 ≤ ∑' j, ‖⟪f, atoms j⟫_ℂ‖ ^ 2 ∧
      ∑' j, ‖⟪f, atoms j⟫_ℂ‖ ^ 2 ≤ B * ‖f‖ ^ 2
\end{leancode}

With these definitions, the translation of the main theorem into Lean is now complete.

\section*{Acknowledgments}

L.L.~is grateful to the Azrieli Foundation for the award of an Azrieli Fellowship and acknowledges the support of this research by ISF Grant No.~854/25.

\bibliographystyle{plain}
\bibliography{bibfile}

\end{document}